\documentclass[a4paper]{scrartcl}

\usepackage{amsmath}
\usepackage{amssymb}
\usepackage[hypertexnames=false]{hyperref}
\usepackage{cleveref}
\usepackage{autonum}
\usepackage[vmargin=3cm]{geometry}
\usepackage[svgnames,dvipsnames]{xcolor}
\usepackage{enumitem}
\usepackage{amsthm}
\usepackage{mathtools}
\usepackage{mleftright}
\usepackage{bm}
\usepackage{xparse}

\hypersetup{
    colorlinks=true,
    linkcolor=RoyalBlue,
    citecolor=Green
}

\colorlet{refkey}{pink!90!red}
\colorlet{labelkey}{JungleGreen!80!yellow}

\newtheorem{theorem}{Theorem}[section]
\crefname{theorem}{Theorem}{Theorems}
\newtheorem{proposition}{Proposition}[section]
\crefname{proposition}{Proposition}{Propositions}
\newtheorem{lemma}{Lemma}[section]
\crefname{lemma}{Lemma}{Lemmas}
\newtheorem{corollary}{Corollary}[section]
\crefname{corollary}{Corollary}{Corollaries}

\theoremstyle{definition}
\newtheorem{remark}{Remark}[section]
\crefname{remark}{Remark}{Remarks}

\crefname{definition}{Definition}{Definitions}

\crefname{section}{Section}{Sections}

\numberwithin{equation}{section}

\numberwithin{equation}{section}
\allowdisplaybreaks

\hypersetup{
    colorlinks=true,
    linkcolor=RoyalBlue,
    citecolor=Green
}

\setlist[enumerate,1]{label=(\arabic*)}

\DeclarePairedDelimiterXPP{\dual}[2]{}{\langle}{\rangle}{}{#1,#2}
\DeclarePairedDelimiterXPP{\Lpnorm}[4]{}{\lVert}{\rVert}{_{L^{#2}(#3; #4)}}{#1}
\DeclarePairedDelimiterXPP{\norm}[2]{}{\lVert}{\rVert}{_{#2}}{#1}

\DeclarePairedDelimiter{\abs}{\lvert}{\rvert}
\DeclarePairedDelimiter{\paren}{\lparen}{\rparen}
\DeclarePairedDelimiter{\bracket}{\lbrack}{\rbrack}

\newcommand{\jump}[1]{[\![#1]\!]}

\providecommand\given{}
\newcommand\SetSymbol[1][]{%
    \nonscript\:#1\vert
    \allowbreak
    \nonscript\:
    \mathopen{}
}
\DeclarePairedDelimiterX{\Set}[1]{\{}{\}}{%
    \renewcommand\given{\SetSymbol[\delimsize]}
    #1
}

\newcommand{\bR}{\mathbb{R}}
\newcommand{\bN}{\mathbb{N}}
\newcommand{\bC}{\mathbb{C}}
\newcommand{\cP}{\mathcal{P}}
\newcommand{\cW}{\mathcal{W}}
\newcommand{\cL}{\mathcal{L}}
\newcommand{\cJ}{\mathcal{J}}
\newcommand{\cT}{\mathcal{T}}
\newcommand{\cJt}{\cJ_\tau}

\newcommand{\brokenSobolev}[1][X]{\cW^{1,1}(\cJt; #1)}

\newcommand{\tiln}{{\tilde{n}}}
\newcommand{\tilt}{{\tilde{t}}}
\newcommand{\regdelta}{\delta^{\tiln,\tilt}}
\newcommand{\reggamma}{\gamma^{\tiln,\tilt}}
\newcommand{\regg}{g^{\tiln,\tilt}}
\newcommand{\reggammatau}{\gamma_\tau^{\tiln,\tilt}}
\newcommand{\reggtau}{g_\tau^{\tiln,\tilt}}
\newcommand{\err}{e^{\tiln,\tilt}}
\newcommand{\interr}{z^{\tiln,\tilt}}
\newcommand{\res}{z_\tau^{\tiln,\tilt}}

\newcommand{\weight}[1][\tilt]{\omega_{#1}}

\newcommand{\initialspace}{{(X,D(A))_{1-\frac{1}{p},p}}}

\title{Discrete maximal regularity for the discontinuous Galerkin time-stepping method without logarithmic factor}

\author{    
    Takahito Kashiwabara%
    \thanks{%
        Graduate School of Mathematical Sciences, the University of Tokyo.
    } {}
    and Tomoya Kemmochi%
    \thanks{%
        Corresponding author. 
        Graduate School of Engineering, Nagoya University. 
        \\
        Email: \texttt{kemmochi@na.nuap.nagoya-u.ac.jp}, Web: \url{https://t-kemmochi.github.io/en/}
    }
}

\hypersetup{
    pdftitle={A sharp discrete maximal regularity for the discontinuous Galerkin time-stepping method},
    pdfauthor={T. Kashiwabara \& T. Kemmochi}
}

\begin{document}
\mleftright

\maketitle

\begin{abstract}
    Maximal regularity is a kind of a priori estimates for parabolic-type equations and it plays an important role in the theory of nonlinear differential equations.
    The aim of this paper is to investigate the temporally discrete counterpart of maximal regularity for the discontinuous Galerkin (DG) time-stepping method.
    We will establish such an estimate without logarithmic factor over a quasi-uniform temporal mesh.
    To show the main result, we introduce the temporally regularized Green's function and then reduce the discrete maximal regularity to a weighted error estimate for its DG approximation.
    Our results would be useful for investigation of DG approximation of nonlinear parabolic problems. 
\end{abstract}

\section{Introduction and results}

Let $X$ be a reflexive Banach space and let $A$ be a densely defined closed linear operator on $X$. 
Then, we consider the abstract Cauchy problem 
\begin{equation}
    \begin{cases}
        \partial_t u(t) + A u(t) = f(t), & t \in (0,T) \eqqcolon J, \\
        u(0) = u_0, 
    \end{cases}
    \label{eq:cauchy}
\end{equation}
where $f \in L^p(J;X)$ and $u_0 \in \initialspace$ are given and $p \in (1,\infty)$. 
Here $D(A)$ denotes the domain of $A$ equipped with the graph norm. 
We assume that $A$ has maximal regularity, which means that there exists a unique mild solution of \eqref{eq:cauchy} satisfying the a priori estimate
\begin{equation}
    \Lpnorm{\partial_t u}{p}{J}{X} + \Lpnorm{A u}{p}{J}{X}
    \le C \paren*{ \Lpnorm{f}{p}{J}{X} + \norm{u_0}{\initialspace} }.
    \label{eq:mr}
\end{equation}
We moreover assume $0 \in \rho(A)$ for simplicity.
Typical examples are elliptic operators of $2m$-th order on $L^q(\Omega)$ (cf.~\cite{MR2006641}) and the Stokes operator on $L^q_\sigma(\Omega)$ (cf.~\cite{zbMATH03627669,MR2602914}), where $\Omega \subset \bR^N$ is a suitable domain and $q \in (1,\infty)$. 
Moreover, under suitable assumptions, finite element approximation of an elliptic operator (say $A_h$) has maximal regularity on the corresponding finite element space (say $X_h$) equipped with the $L^q$-norm for $q \in (1,\infty)$ uniformly in $h$ (see \cite{MR2218965,MR3614012,MR3854049,MR3395142,MR3778340,MR4081914}).

It is known that an operator having maximal regularity generates a bounded analytic semigroup on $X$.
In particular, together with $0 \in \rho(A)$, there exists $\delta_0 \in (0, \pi/2 )$ such that the resolvent set $\rho(A)$ includes a sector domain $\Sigma_{\delta_0} \coloneqq \Set{z \in \bC \setminus \{0\} \given |\arg z| > \delta_0}$ and the resolvent estimate
\begin{equation}
    \norm{(\lambda I - A)^{-1}v}{X} \le C (1+|\lambda|)^{-1} \norm{v}{X}   
    \label{eq:resolvent}
\end{equation} 
holds for all $v \in X$ and $\lambda \in \Sigma_{\delta_0}$.
Maximal regularity plays an important role in the theory of mild solutions to nonlinear partial differential equations.
For more details of maximal regularity and its application, we refer the reader to \cite{MR2006641} and references therein.

The aim of this paper is to investigate the temporally discrete counterpart of maximal regularity.
As a discretization method, we focus on the discontinuous Galerkin (DG) time-stepping method for \eqref{eq:cauchy}.
Let $\cJt = \{ J_n \}_{n=1}^N$ be a temporal mesh of the interval $J$ with $N$ pieces defined by
\begin{equation}
    J_n = (t_{n-1}, t_n), \qquad 0 = t_0 < t_1 < \dots < t_N = T,
\end{equation}
and we set $\tau_n = t_n-t_{n-1}$ and $\tau = \max_n \tau_n$.
We assume that the family of meshes $\{ \cJt \}_\tau$ is quasi-uniform. Namely, there exists $C>0$ such that, for any mesh $J_\tau$, 
\begin{equation}
    \tau_n \ge C \tau, \qquad \forall J_n \in \cJt.
\end{equation}
We introduce the $X$-valued broken Sobolev space associated to the mesh $\cJt$ by
\begin{equation}
    \brokenSobolev \coloneqq \Set{v \in L^1_\mathrm{loc}(J;X) \given v|_{J_n} \in W^{1,1}(J_n;X), \, \forall J_n \in \cJt}.
\end{equation}
and for $v \in \brokenSobolev$ we use the following notation: 
\begin{equation}
    v^{n,\pm} \coloneqq \lim_{s \downarrow 0} v(t_n \pm s), \quad \jump{v}^n \coloneqq v^{n,+} - v^{n,-}.
\end{equation}
We then define the space of $X$-valued piecewise polynomials of degree $r$ associated to $\cJt$ by
\begin{align}
    S^n_\tau(X) &\coloneqq \cP^r(J_n; X) \coloneqq \Set*{v_\tau = \sum_{j=0}^r v_j t^j \given v_j \in X ,\,j=0,\dots,r}, \\ 
    S_\tau(X) &\coloneqq \Set*{v_\tau \in \brokenSobolev \given v_\tau|_{J_n} \in S^n_\tau(X) ,\,n=1,\dots,N}.
\end{align} 
We write $\cP^r(J') \coloneqq \cP^r(J'; \bR)$ for an interval $J' \subset \bR$ for brevity. 
Finally, we denote the dual space of $X$ by $X'$ and the duality pairing by $\dual{v}{\varphi}$ for $v \in X$ and $\varphi \in X'$.
We put the element of $X'$ on the right side, because $X'$ is considered as the space of test functions throughout this paper.

The DG time-stepping method for \eqref{eq:cauchy} is now formulated as follows (cf.~\cite{MR826227,MR2249024}).
Find $u_\tau \in S_\tau(D(A))$ that satisfies
\begin{equation}
    \begin{dcases}
        \sum_{n=1}^N \int_{J_n} \dual*{\partial_t u_\tau + Au_\tau}{\varphi_\tau} dt + \sum_{n=1}^N \dual*{\jump{u_\tau}^{n-1}}{\varphi_\tau^{n-1,+}} 
        = \int_{J} \dual*{f}{\varphi_\tau} dt,
        & \forall \varphi_\tau \in S_\tau(X'), \\[1ex]
        u_\tau^{0,-} = u_0.
    \end{dcases}
    \label{eq:DG}
\end{equation} 
By the discontinuity of the test function $\varphi_\tau$, the solution $u_\tau$ satisfies
\begin{equation}
    \int_{J_n} \dual*{\partial_t u_\tau + Au_\tau}{\varphi_\tau} dt + \dual*{\jump{u_\tau}^{n-1}}{\varphi_\tau^{n-1,+}} = \int_{J_n} \dual*{f}{\varphi_\tau} dt,
    \quad \forall \varphi_\tau \in S^n_\tau(X'),
    \label{eq:DG-n}
\end{equation} 
for all $n=1,\dots,N$. 
The well-posedness of \eqref{eq:DG} can be shown by the argument similar to that of \cite[p.~1322]{MR1620144} since $A$ generates a bounded analytic semigroup.
Moreover, focusing on the sequence $(u_\tau^{n,-})_{n=0}^N$ only, one may regard the DG time-stepping method as a one-step method, which is known to be (strongly) A-stable.
We refer the reader to \cite{MR826227,MR2249024} for the detail of theoretical aspects of the DG time stepping method.

The main result of the present paper is to show the following estimate that corresponds to the maximal regularity \eqref{eq:mr}.
\begin{theorem}[Discrete maximal regularity]
    \label{thm:dmr} 
    Let $X$ be a reflexive Banach space, $A$ be a densely defined closed linear operator on $X$, 
    $p \in (1,\infty)$, $f \in L^p(J;X)$, and $u_0 \in \initialspace$.
    Assume that $0 \in \rho(A)$, $A$ has maximal regularity, the family of meshes $\{ \cJt \}_\tau$ is quasi-uniform, and $r \ge 1$.
    Then, the solution $u_\tau \in S_\tau(D(A))$ of \eqref{eq:DG} satisfies the discrete maximal regularity
    \begin{multline}
        \label{eq:dmr}
        \paren*{ \sum_{n=1}^N \Lpnorm{\partial_t u_\tau}{p}{J_n}{X}^p }^{1/p} 
        + \Lpnorm{Au_\tau}{p}{J}{X} 
        + \paren*{\sum_{n=1}^N \norm*{\frac{\jump{u_\tau}^{n-1}}{\tau_n}}{X}^p \tau_n }^{1/p}  \\ 
        \le C \paren*{ \Lpnorm{f}{p}{J}{X} + \norm{u_0}{\initialspace} },
    \end{multline}
    where $C$ is independent of $\tau$, $f$, and $u_0$.
\end{theorem} 
The assumption $r \ge 1$ is essential in the proof of the above result (see \cref{rem:r}).
Hence the case $r=0$ is not covered, although in this case the scheme \eqref{eq:DG} becomes a variant of the backward Euler method.

Temporal discretization of maximal regularity have been studied in recent decades (e.g.~\cite{MR1299329,MR1853519,MR3244339,MR3582825,MR3549514,MR3606467,MR4410757,MR4368992}) and is applied to numerical analysis of linear and nonlinear parabolic problems (e.g.~\cite{MR3498514,MR3620143,MR3683415,MR3778340,MR3857906,MR4246878}).
For the DG time-stepping method, since it is A-stable as a one-step method as mentioned above, it has (one-step version of) discrete maximal regularity by the result of \cite{MR3582825} if the temporal mesh is uniform and the Banach space $X$ is UMD.

The estimates over the intervals $J_n$ are investigated in \cite{MR3498514,MR3606467,MR4368992}. 
In \cite{MR3498514,MR3606467}, the estimates of the form \eqref{eq:dmr} with $C$ replaced by $C |\log (T/\tau)|$ are established with locally quasi-uniform temporal meshes, when $A$ is an elliptic operator and $X=L^q(\Omega)$ for $q \in [1,\infty]$. 
The approach of \cite{MR3498514,MR3606467} does not rely on the maximal regularity for $A$ (but on semigroup property only).
This allows them to consider the end-point cases ($X=L^1(\Omega)$ or $L^\infty(\Omega)$), where the appearance of the logarithmic factor is natural.
However, the logarithmic-free estimate is better for the reflexive case (i.e., $X=L^q(\Omega)$ with $1 < q < \infty$). 

To remove the logarithmic factor, the authors of \cite{MR4368992} considered the reconstruction of $u_\tau$, which is originally introduced in \cite{MR2249675}.
By regarding the DG time-stepping method as a modified Radau IIA method, they established the estimate \eqref{eq:dmr} with $u_\tau$ replaced by the reconstruction when the temporal mesh is uniform and $X$ is UMD. 
Although this result was successfully applied to a posteriori error estimates, the estimate of the form \eqref{eq:dmr} was still open.

Our approach to establish the estimate \eqref{eq:dmr} is different from both of the above approaches.
In order to show the discrete maximal regularity, we introduce the \emph{temporally} regularized Green's function.
Fix an interval $J_\tiln$ and time $\tilt \in J_\tiln$ arbitrarily and introduce a regularized delta function $\regdelta(t)$ depending only on the time variable $t$.
Precise definition will be given in Section~\ref{sec:green}.
Then, we consider the solution of the dual parabolic equation 
\begin{equation}
    \begin{cases}
        -\partial_t \reggamma(t) + A'\reggamma(t) = \regdelta(t) \varphi, & t \in J, \\ 
        \reggamma(T) = 0
    \end{cases}
\end{equation}
with $\varphi \in X'$ arbitrary, where $A'$ is the dual operator of $A$ defined on $X'$.
We call $\reggamma$ the temporally regularized Green's function, and we will show that the discrete maximal regularity \eqref{eq:dmr} is reduced to a weighted error estimate for the DG approximation of $\reggamma$ (see~\cref{lem:dmr-to-weighted}).

The key point to show the weighted error estimate is to split the error $\reggamma - \reggammatau$ into the interpolation error $\reggamma - I_\tau \reggamma$ and the ``discrete error'' $I_\tau \reggamma - \reggammatau$ as in the literature, where $\reggammatau \in S_\tau(D(A'))$ is the DG approximation of $\reggamma$ and $I_\tau$ is a suitable interpolation operator.
Then, it is known that the discrete error satisfies \eqref{eq:DG} with $f$ replaced by $-A'(\reggamma - I_\tau \reggamma) $ (cf.~\cite[p.~208]{MR2249024}),
which allows us to obtain an expression formula for the discrete error in terms of the interpolation error (cf.~\cite{MR3606467}).
We then investigate the rational functions appearing in the formula more rigorously than the literature (such as \cite{MR1620144}) and finally obtain the weighted error estimate.

In the end of this paper, we will present three simple corollaries.
We will show that the estimate \eqref{eq:dmr} includes the one-step version of discrete maximal regularity.
Moreover, an optimal order error estimate (without logarithmic factor) for \eqref{eq:DG} will be derived by the same approach as in \cite{MR3606467}.
Finally, we mention the fully discrete case. 
Since $X$ and $A$ are not limited to the continuous case, fully discrete maximal regularity for a parabolic problem is immediately obtained, if a discrete elliptic operator $A_h$ has maximal regularity on the discrete space $X_h$ uniformly in $h$.
Optimal order error estimate is also obtained by the same proof as in \cite{MR3606467}.
Further applications, such as application to nonlinear problems, would be investigated in the future.

The rest of the present paper is organized as follows.
In Section~\ref{sec:preliminaries}, we collect some preliminary results related to the DG method.
Moreover, we will show that the estimate for $Au_\tau$ is essential to show \eqref{eq:dmr} in the last part of this section.
In Section~\ref{sec:green}, we will introduce the temporally regularized Green's function and show that \eqref{eq:dmr} is reduced to a weighted error estimate for $\reggamma - \reggammatau$.
The proof of the main result is given in Section~\ref{sec:proof}.
We first show the case of $u_0 = 0$ via the temporally regularized Green's function, and then consider the initial value problem with $f = 0$.
Finally, in Section~\ref{sec:application}, we mention the relation between our estimate and existing discrete maximal regularity, and we show an optimal order error estimate for both the semi- and fully-discrete cases.

\section{Preliminaries}
\label{sec:preliminaries}

We collect some preliminary estimates for the proof of \eqref{eq:dmr}.
Throughout this paper, the symbol $C$ stands for a general constant independent of $\tau$ and other parameters.
Its value may be different at each appearance.

\subsection{Properties of the DG time-stepping method}

Define a bilinear form $B_\tau$ by
\begin{equation}
    B_\tau(v,\varphi) 
    \coloneqq \sum_{n=1}^N \int_{J_n} \dual*{\partial_t v + Av}{\varphi} dt 
    + \sum_{n=2}^N \dual*{\jump{v}^{n-1}}{\varphi^{n-1,+}}
    + \dual*{v^{0,+}}{\varphi^{0,+}}
\end{equation}
for $v \in \brokenSobolev[D(A)]$ and $\varphi \in \brokenSobolev[X']$. 
Then, the DG scheme \eqref{eq:DG} is equivalent to
\begin{equation}
    B_\tau(u_\tau,\varphi_\tau) = \int_J \dual*{f}{\varphi_\tau} dt + \dual*{u_0}{\varphi_\tau^{0,+}}, \qquad \forall \varphi_\tau \in S_\tau(X).
\end{equation}
Moreover, the solution $u \in L^p(J; D(A)) \cap W^{1,p}(J; X)$ of the Cauchy problem \eqref{eq:cauchy} satisfies
\begin{equation}
    B_\tau(u,\varphi) = \int_J \dual*{f}{\varphi} dt + \dual*{u_0}{\varphi^{0,+}}, \qquad \forall \varphi \in \brokenSobolev[X'],
\end{equation}
which implies the Galerkin orthogonality (compatibility)
\begin{equation}
    B_\tau(u-u_\tau,\varphi_\tau) = 0, \qquad \forall \varphi_\tau \in S_\tau(X).
\end{equation}

Denote the dual operator of $A$ by $A'$ and define the dual form of $B_\tau$ by
\begin{equation}
    B'_\tau(v,\varphi) 
    \coloneqq \sum_{n=1}^N \int_{J_n} \dual*{v}{-\partial_t \varphi + A' \varphi} dt 
    - \sum_{n=1}^{N-1} \dual*{v^{n,-}}{\jump{\varphi}^n} 
    + \dual*{v^{N,-}}{\varphi^{N,-}}
\end{equation}
for $v \in \brokenSobolev$ and $\varphi \in \brokenSobolev[D(A')]$. 
Then, it is clear that 
\begin{equation}
    B_\tau(v,\varphi) = B'_\tau(v,\varphi)
    \label{eq:duality}
\end{equation}
for any $v \in \brokenSobolev[D(A)]$ and $\varphi \in \brokenSobolev[D(A')]$.

It is known that the solution of \eqref{eq:DG} is expressed by using rational approximations of the semigroup $e^{-tA}$.
Let $\{ \phi_j \}_{j=0}^r \subset \cP^r(0,1)$ be the Lagrange basis functions with respect to the nodal points $\{ j/r \}_{j=0}^r$, and let $\phi^n_j(t) = \phi_j((t-t_{n-1})/\tau_n)$.
We let $u_\tau \in S_\tau(X)$ be the solution of \eqref{eq:DG} and express $u_\tau|_{J_n} \in \cP^r(J_n;X)$ as
\begin{equation}
    u_\tau(t) = \sum_{j=0}^r U^n_j \phi^n_j(t), \qquad U^n_j \in X, \quad t \in J_n.
    \label{eq:local-basis}
\end{equation}
Then, $U^n_j$ has the following expression (cf.~\cite[p.~1322]{MR1620144}, \cite[\S~4.2]{MR3606467}).

\begin{lemma}\label{lem:expression-DG}
    Let $U^n_j$ as above. Then, there exist polynomials $\hat{q}$ and $q_{i,j}$ ($i,j=0,\dots,r$) that satisfy
    \begin{equation}
        U^n_i = R_{i,0}(\tau_n A) u_\tau^{n-1} + \sum_{j=0}^r R_{i,j}(\tau_n A) \int_{J_n} f(t) \phi^n_j(t) dt,
        \qquad R_{i,j} = \hat{q}^{-1} q_{i,j},
        \label{eq:expression-DG}
    \end{equation}
    for any $n=1,\dots,N$.
    Moreover, $\hat{q}$, $q_{i,j}$, and $R_{i,j}$ have the following properties:
    \begin{enumerate}[label=(\arabic*),font=\upshape]
        \item $\deg \hat{q} = r+1$, $\deg q_{i,j} \le r$
        \item $R_{i,j}$ is holomorphic and bounded in the right half plain of $\bC$
        \item For $\delta \in (\delta_0,\pi/2)$ and $\lambda \in \Gamma_\delta \coloneqq \partial \Sigma_\delta$,
        \begin{equation}
            |R_{i,0}(\lambda)| \le \frac{1}{1+C |\lambda|}, \quad  
            |R_{i,j}(\lambda)| \le \frac{C}{1+C|\lambda|},
            \qquad 
            i,j = 0,\dots,r
            \label{eq:rational-DG}
        \end{equation}
        hold with $C$ depending only on $r$ and $\delta$, where $\delta_0$ is as in \eqref{eq:resolvent}.
    \end{enumerate}
\end{lemma}

Using \eqref{eq:expression-DG} repeatedly, we obtain the following Duhamel-type formula.
\begin{corollary}\label{cor:Duhamel-DG}
    For $n=1,\dots,N$ and $i=0,\dots,r$, we have
    \begin{multline}
        U^n_i 
        = R_{i,0}(\tau_n A) \paren*{ \prod_{l=1}^{n-1} R_{r,0}(\tau_l A) } u_0  \\ 
        + \sum_{m=1}^{n-1} R_{i,0}(\tau_n A) \paren*{ \prod_{l=m+1}^{n-1} R_{r,0}(\tau_l A) } \sum_{j=0}^r R_{r,j}(\tau_m A) \int_{J_m} f(t) \phi^m_j(t) dt \\ 
        + \sum_{j=0}^r R_{i,j}(\tau_n A) \int_{J_n} f(t) \phi^n_j(t) dt.
    \end{multline}
    Here, we set $\sum_{m=1}^0 = 0$ and $\prod_{l=n}^{n-1} = 1$.
\end{corollary}

\subsection{Interpolation operator}

Define $I_\tau^n \colon W^{1,1}(J_n;X ) \to \cP^r(J_n; X)$ by
\begin{align}
    &(I_\tau v)^{n,-} = v(t_n), \\ 
    &\int_{J_n} (I_\tau v - v) q dt = 0, \quad \forall q \in \cP^{r-1}(J_n) 
\end{align}
(cf.~\cite[p.~207]{MR2249024}) and define an ``interpolation'' operator $I_\tau \colon W^{1,1}(J; X) \to S_\tau(X)$ by
\begin{equation}
    (I_\tau v)|_{J_n} = I_\tau^n v, \qquad \forall n=1,\dots,N.
\end{equation}
Then, the following error estimate holds without quasi-uniformity.
\begin{lemma}
    Let $r \ge 1$, $p \in [1,\infty]$, and $0 \le l \le m \le r$. Then, the error estimate 
    \begin{equation}
        \norm{\partial_t^l (I_\tau^n v - v)}{L^p(J_n; X)}
        \le C \tau_n^{m+1-l} \norm{\partial_t^{m+1} v}{L^p(J_n;X)},
        \qquad v \in W^{m+1}(J_n; X)
        \label{eq:int-err}
    \end{equation}
    holds.
\end{lemma}

The proof is standard; however, 
we here give a proof for the reader's convenience.

\begin{proof}
    Define an operator $\hat{I} \colon W^{1,1}(0,1;X) \to \cP^r(0,1;X)$ on the interval $[0,1]$ by
    \begin{equation}
        (\hat{I}v)(0) = v(0),
        \qquad 
        \int_0^1 (\hat{I}v - v) t^j dt = 0, \quad j=0,\dots,r-1.
    \end{equation}
    Then, it suffices to show
    \begin{equation}
        \norm{\partial_t^l (\hat{I} v - v)}{L^p(0,1; X)}
        \le C \norm{\partial_t^{m+1} v}{L^p(0,1;X)},
        \qquad v \in W^{m+1}(0,1; X)
        \label{eq:int-err-ref}
    \end{equation}
    by change of variables.
    Notice that $\hat{I}q = q$ for $q \in \cP^r(0,1;X)$ (see \cite[p.~208]{MR2249024}).

    Let us prove \eqref{eq:int-err-ref}.
    The operator $\hat{I}$ is bounded as a map from $W^{m+1,p}$ to $W^{l,p}$ by the closed graph theorem.
    Thus the operator $L \coloneqq \hat{I} - \mathrm{id} \colon W^{m+1,p} \to W^{l,p}$ is also bounded, where $\mathrm{id}$ is the identity map.
    Therefore, for any $q \in \cP^r(0,1;X)$, we have 
    \begin{equation}
        \norm{\partial_t^l (\hat{I} v - v)}{L^p(0,1; X)}
        = \norm{\partial_t^l L (v-q)}{L^p(0,1; X)}
        \le \norm{L}{} \norm{v-q}{W^{m+1,p}(0,1;X)}
    \end{equation}
    owing to $Lq = 0$, which implies \eqref{eq:int-err-ref} by the Taylor theorem.
\end{proof}

\subsection{The estimates for time derivative and jump}

The purpose of this subsection is to show that the estimates of $\partial_t u_\tau$ and $\jump{u_\tau}^{n-1}$ of the discrete maximal regularity \eqref{eq:dmr} are reduced to that of $Au_\tau$.

\begin{lemma}\label{lem:dudt}
    Let $p \in [1,\infty]$ and let $u_\tau \in S_\tau(D(A))$ be the solution of \eqref{eq:DG}.
    Then, for any $n = 1,\dots,N$, we have
    \begin{equation}
        \Lpnorm{\partial_t u_\tau}{p}{J_n}{X} \le C \Lpnorm{A u_\tau - f}{p}{J_n}{X},
    \end{equation}
    where $C$ is independent of $n$ and $\tau$.
\end{lemma}

\begin{proof}
    Let $\{ \hat{\phi}_j \}_{j=0}^{r-1} \subset \cP^{r-1}$ an orthonormal system satisfying
    \begin{equation}
        \int_0^1 t \hat{\phi}_i(t) \hat{\phi}_j(t) dt = \delta_{i,j}
    \end{equation}
    and let
    \begin{equation}
        \phi_j^n(t) = \frac{1}{\tau_n} \hat{\phi}_j \paren*{\frac{t-t_{n-1}}{\tau_n}}.
    \end{equation}
    Then, the system $\{ \phi_j^n \}_j$ satisfies
    \begin{equation}
        \int_{J_n} (t-t_{n-1}) \phi_i^n(t) \phi_j^n(t) dt = \delta_{i,j}
    \end{equation}
    and
    \begin{equation}
        \norm{\phi_j^n}{L^q(J_n)} \le C \tau_n^{-1+\frac{1}{q}}, \qquad q \in [1,\infty]
    \end{equation}
    by scaling.
    Let us now express $\partial_t u_\tau|_{J_n} \in \cP^{r-1}(J_n;X)$ by
    \begin{equation}
        \partial_t u_\tau = \sum_{j=0}^{r-1} V_j \phi_j^n, \quad V_j \in X.
    \end{equation}
    Then we have
    \begin{equation}
        \Lpnorm{\partial_t u_\tau}{p}{J_n}{X}
        \le \sum_j \norm{V_j}{X} \norm{\phi_j^n}{L^p(J_n)}
        \le C \tau_n^{-1+\frac{1}{p}} \sum_j \norm{V_j}{X} .
        \label{eq:dudt-1}
    \end{equation}

    Moreover, fix any $V'_j \in X'$ satisfying
    \begin{equation}
        \dual*{V_j}{V'_j} = \norm{V_j}{X}, \qquad \norm{V'_j}{X'} = 1
    \end{equation}
    for $j=0,\dots,r$ and set 
    \begin{equation}
        \Psi =\sum_{j=0}^{r-1} V'_j \phi_j^n \in \cP^{r-1}(J_n;X').
    \end{equation}
    Then, substituting $\Phi \coloneqq (t-t_{n-1}) \Psi \in \cP^r(J_n;X')$ as a test function of  \eqref{eq:DG-n}, we have
    \begin{equation}
        \int_{J_n} \dual*{\partial_t u_\tau}{\Phi} dt = - \int_{J_n} \dual*{A u_\tau - f}{\Phi} dt 
        \le \Lpnorm{A u_\tau - f}{p}{J_n}{X} \Lpnorm{\Phi}{p'}{J_n}{X'}
    \end{equation}
    since $\Phi^{n-1}=0$.
    Therefore, noticing 
    \begin{equation}
        \int_{J_n} \dual*{\partial_t u_\tau}{\Phi} dt
        = \sum_{i,j} \int_{J_n} (t-t_{n-1}) \phi_i^n \phi_j^n dt \dual*{V_i}{V'_j}
        = \sum_j \norm{V_j}{X}
    \end{equation}
    and
    \begin{equation}
        \Lpnorm{\Phi}{p'}{J_n}{X'}
        \le \tau_n \sum_j \norm{V'_j}{X'} \norm{\phi_j^n}{L^{p'}(J_n)}
        \le C \tau_n^{1-\frac{1}{p}},
    \end{equation}
    we obtain
    \begin{equation}
        \sum_j \norm{V_j}{X}
        \le C \tau_n^{1-\frac{1}{p}} \Lpnorm{A u_\tau - f}{p}{J_n}{X}.
        \label{eq:dudt-2}
    \end{equation}
    Hence we establish the desired estimate from \eqref{eq:dudt-1} and \eqref{eq:dudt-2}.
\end{proof}

\begin{lemma}\label{lem:jump}
    Let $p \in [1,\infty]$ and let $u_\tau$ be the solution of \eqref{eq:DG}.
    Then, for any $n=1,\dots,N$, we have
    \begin{equation}
        \norm*{\frac{\jump{u_\tau}^{n-1}}{\tau_n}}{X} \tau_n^{1/p} \le \Lpnorm{\partial_t u_\tau + A u_\tau - f}{p}{J_n}{X}.
    \end{equation}
\end{lemma}

\begin{proof}
    Let $\varphi \in X'$ be such that
    \begin{equation}
        \dual*{\jump{u_\tau}^{n-1}}{\varphi} = \norm*{\jump{u_\tau}^{n-1}}{X},
        \qquad 
        \norm{\varphi}{X'} = 1.
    \end{equation}
    Then, substituting $\varphi_\tau|_{J_n} \equiv \varphi \in \cP^0(J_n;X') \subset S_\tau^n(X')$ into \eqref{eq:DG-n}, we have
    \begin{align}
        \norm*{\jump{u_\tau}^{n-1}}{X}
        &= - \int_{J_n} \dual*{\partial_t u_\tau + A u_\tau - f}{\varphi} dt  \\ 
        &\le \Lpnorm{\partial_t u_\tau + A u_\tau - f}{p}{J_n}{X} \Lpnorm{\varphi}{p'}{J_n}{X'} \\ 
        &= \Lpnorm{\partial_t u_\tau + A u_\tau - f}{p}{J_n}{X} \tau_n^{1/p'} ,
    \end{align}
    which itself is the desired estimate.
\end{proof}

Thanks to \cref{lem:dudt,lem:jump}, it suffices to show 
\begin{equation}
    \Lpnorm{Au_\tau}{p}{J}{X} 
    \le C \paren*{ \Lpnorm{f}{p}{J}{X} + \norm{u_0}{\initialspace} }
    \label{eq:dmr-Au}
\end{equation}
for the proof of \eqref{eq:dmr}.
The subsequent sections are devoted to show this estimate.

\section{Temporally regularized Green's function}
\label{sec:green}

In this section, we introduce the temporally regularized Green's function, which plays an important role for the proof of \eqref{eq:dmr-Au}.

\subsection{Temporally regularized delta and Green's functions}

Hereafter, we fix an interval $J_\tiln \in \cJt$ and a time $\tilt \in J_\tiln$ arbitrarily.
Then, we can construct a regularized delta function $\regdelta \in C_0^\infty(J_\tiln)$ that satisfies
\begin{equation}
    \int_{J_\tiln} \regdelta(t) q(t) dt = q(\tilt), \qquad \forall q \in \cP^r(J_\tiln).
\end{equation}
The norms of $\regdelta$ have the upper bounds
\begin{equation}
    \norm{\partial_t^l \regdelta}{L^p(J_n)} \le C \tau^{-l-1+\frac{1}{p}},
    \label{eq:delta-norm}
\end{equation}
for any $l=0,1,\dots$ and $p \in [1,\infty]$, where $C$ is independent of $\tau$, $\tiln$, and $\tilt$. 
Construction of such $\regdelta$ can be found in \cite[p.~520]{MR1388486}. 

Fix $\varphi \in X'$ arbitrarily. Then, we define the temporally regularized Green's function $\reggamma$ as the solution of the dual problem
\begin{equation}
    \begin{cases}
        -\partial_t \reggamma(t) + A'\reggamma(t) = \regdelta(t) \varphi, & t \in J, \\ 
        \reggamma(T) = 0.
    \end{cases}
    \label{eq:cauchy-dual}
\end{equation}
Moreover, let $\reggammatau \in S_\tau(D(A'))$ be the DG approximation of $\reggamma$ satisfying
\begin{equation}
    \begin{dcases}
        B'_\tau(v_\tau,\reggammatau) = \int_J \dual*{v_\tau}{\regdelta\varphi} dt, & \forall v_\tau \in S_\tau(X), \\
        (\reggammatau)^{N,+} = 0.
    \end{dcases}
    \label{eq:DG-dual}
\end{equation}
Since $\reggamma$ satisfies compatibility property
\begin{equation}
    B'_\tau(v,\reggammatau) = \int_J \dual*{v}{\regdelta\varphi} dt,
    \qquad \forall v \in \brokenSobolev,
\end{equation}
we have the Galerkin orthogonality
\begin{equation}
    B'_\tau(v_\tau,\reggamma-\reggammatau) = 0, \qquad \forall v_\tau \in S_\tau(X).
\end{equation}

\subsection{Reduction to weighted estimates for regularized Green's function}

In this section, we address the estimate \eqref{eq:dmr-Au} with $u_0 = 0$ and show that the proof is reduced to a weighted error estimate for $\reggammatau$.
Hereafter, we define a weight function $\weight \in C^\infty(\bR)$ by
\begin{equation}
    \weight(t) \coloneqq \sqrt{(t-\tilt)^2 + \tau^2}, \qquad t \in \bR.
\end{equation}

\begin{lemma}\label{lem:dmr-to-weighted} 
    Let $X$ and $A$ be as in \cref{thm:dmr}. 
    Let $p \in (1,\infty)$ and let $u_\tau \in S_\tau(D(A))$ be the solution of \eqref{eq:DG} with $u_0=0$.
    Set 
    \begin{equation}
        M_{p',\alpha} \coloneqq \sup_{\tiln \in \bN,\tilt \in J_\tiln,\norm{\varphi}{X'}\le 1} \Lpnorm{\weight^\alpha A' (\reggamma - \reggammatau)}{p'}{J}{X'}
    \end{equation}
    for $\alpha > 0$.
    Then, provided that $\alpha > 1/p$, we have 
    \begin{equation}
        \Lpnorm{Au_\tau}{p}{J}{X}
        \le C \paren*{ 1 + M_{p',\alpha} \tau^{-\alpha+\frac{1}{p}} } \Lpnorm{f}{p}{J}{X},
    \end{equation}
    where $C$ is independent of $\tau$, $f$, $\tiln$, and $\tilt$.
\end{lemma}

\begin{proof}
    We address $\dual*{Au_\tau(\tilt)}{\varphi} $ for $\tilt \in J_\tiln$ and $\varphi \in X'$.
    Recalling that $u_0=u_\tau^{0,-}=0$, we have
    \begin{align}
        \dual*{Au_\tau(\tilt)}{\varphi} 
        &= \int_{J_\tiln} \dual*{Au_\tau}{\regdelta\varphi} dt && \\
        &= \int_J \dual*{Au_\tau}{\regdelta\varphi} dt         && \\ 
        &= B'_\tau(Au_\tau,\reggammatau)                       && \text{by \eqref{eq:DG-dual}}\\
        &= B_\tau(u_\tau,A'\reggammatau)                       && \text{by \eqref{eq:duality}}\\
        &= \int_J \dual*{f}{A'\reggammatau} dt                 && \text{by \eqref{eq:DG}}\\
        &= \int_J \dual*{f}{A'\reggamma} dt - \int_J \dual*{f}{A'(\reggamma - \reggammatau)} dt .
    \end{align}
    Moreover, we have 
    \begin{align}
        \int_J \dual*{f}{A'\reggamma} dt
        &= \int_J \dual*{\partial_t u + A u}{A' \reggamma} dt         && \text{by \eqref{eq:cauchy}}\\ 
        &= \int_J \dual*{Au}{-\partial_t \reggamma + A' \reggamma} dt && \\ 
        &= \int_J \dual*{Au}{\regdelta\varphi} dt,                     && \text{by \eqref{eq:cauchy-dual}}
    \end{align}
    since $u(0) = \gamma(T) = 0$.
    We thus set 
    \begin{equation}
        I_1(\tilt;\varphi) = \int_J \dual*{Au}{\regdelta\varphi} dt,
        \qquad 
        I_2(\tilt;\varphi) = \int_J \dual*{f}{A'(\reggamma - \reggammatau)} dt 
    \end{equation}
    and then clearly
    \begin{equation}
        \dual*{Au_\tau(\tilt)}{\varphi} = I_1(\tilt;\varphi) - I_2(\tilt;\varphi) 
    \end{equation}
    holds, which implies
    \begin{multline}
        \Lpnorm{Au_\tau}{p}{J}{X} 
        \le \paren*{ \sum_{\tiln=1}^N \int_{J_\tiln} \sup_{\norm{\varphi}{X'} \le 1} |I_1(\tilt;\varphi)|^p d\tilt }^{1/p} \\
        +   \paren*{ \sum_{\tiln=1}^N \int_{J_\tiln} \sup_{\norm{\varphi}{X'} \le 1} |I_2(\tilt;\varphi)|^p d\tilt }^{1/p}.
    \end{multline}

    Let us address $I_1$. By \eqref{eq:delta-norm}, we have
    \begin{align}
        I_1(\tilt;\varphi) 
        &\le \Lpnorm{Au}{p}{J_\tiln}{X} \Lpnorm{\regdelta \varphi}{p'}{J_\tiln}{X'} \\
        &\le C \tau^{-\frac{1}{p}} \norm{\varphi}{X'} \Lpnorm{Au}{p}{J_\tiln}{X},
    \end{align}
    which implies 
    \begin{equation}
        \sum_{\tiln=1}^N \int_{J_\tiln} \sup_{\norm{\varphi}{X'} \le 1} |I_1(\tilt;\varphi)|^p d\tilt
        \le C \sum_{\tiln=1}^N \int_{J_\tiln} \tau^{-1} \Lpnorm{Au}{p}{J_\tiln}{X}^p d\tilt 
        \le C \Lpnorm{Au}{p}{J}{X}^p .
    \end{equation}
    Therefore, we have 
    \begin{equation}
        \paren*{ \sum_{\tiln=1}^N \int_{J_\tiln} \sup_{\norm{\varphi}{X'} \le 1} |I_1(\tilt;\varphi)|^p d\tilt }^{1/p}
        \le C \Lpnorm{f}{p}{J}{X}
    \end{equation}
    owing to the maximal regularity \eqref{eq:mr}.

    Let us then treat $I_2$. It is clear that
    \begin{equation}
        I_2(\tilt;\varphi) \le \Lpnorm{\weight^{-\alpha} f}{p}{J}{X} \Lpnorm{\weight^\alpha A'(\reggamma - \reggammatau)}{p'}{J}{X'} 
        \le M_{\alpha,p'} \Lpnorm{\weight^{-\alpha} f}{p}{J}{X},
    \end{equation}
    which implies
    \begin{equation}
        \sum_{\tiln=1}^N \int_{J_\tiln} \sup_{\norm{\varphi}{X'} \le 1} |I_2(\tilt;\varphi)|^p d\tilt
        \le M_{\alpha,p'}^p \sum_{\tiln=1}^N \int_{J_\tiln} \Lpnorm{\weight^{-\alpha} f}{p}{J}{X}^p d\tilt .
    \end{equation}
    Observe that
    \begin{align}
        \sum_{\tiln=1}^N \int_{J_\tiln} \Lpnorm{\weight^{-\alpha} f}{p}{J}{X}^p d\tilt
        &= \int_J \int_J \weight(t)^{-\alpha p} \norm{f(t)}{X}^p dt d\tilt \\ 
        &= \int_J \norm{f(t)}{X}^p \int_J \paren*{(t-\tilt)^2 + \tau^2}^{-\alpha p/2} d\tilt dt.
    \end{align}
    Now, by rescaling and $\alpha > 1/p \implies -\alpha p + 1<0$, one has
    \begin{equation}
        \int_J \paren*{(t-\tilt)^2 + \tau^2}^{-\alpha p/2} d\tilt
        \le 2\tau^{-\alpha p + 1} \int_0^\infty (s^2 + 1)^{-\alpha p/2} ds
        \le C \tau^{-\alpha p + 1}.
    \end{equation}
    Therefore, we have
    \begin{equation}
        \sum_{\tiln=1}^N \int_{J_\tiln} \Lpnorm{\weight^{-\alpha} f}{p}{J}{X}^p d\tilt
        \le C \tau^{-\alpha p+ 1} \Lpnorm{f}{p}{J}{X}^p
    \end{equation}
    and thus obtain 
    \begin{equation}
        \paren*{ \sum_{\tiln=1}^N \int_{J_\tiln} \sup_{\norm{\varphi}{X'} \le 1} |I_2(\tilt;\varphi)|^p d\tilt }^{1/p}
        \le C M_{\alpha,p'} \tau^{-\alpha + \frac{1}{p}} \Lpnorm{f}{p}{J}{X}.
    \end{equation}
    Hence we complete the proof owing to the reflexivity of $X$. 
\end{proof}

By the change of variable $t \mapsto T-t$ and by replacing $(X',A',p')$ by $(X,A,p)$, it suffices to show the following assertion for the proof of \eqref{eq:dmr-Au} with $u_0 = 0$.

\begin{proposition}\label{prop:weighted-error-X}
    Let $p \in (1,\infty)$.
    Fix $v \in X$, $\tiln \in \bN$, and $\tilt \in J_\tiln$ arbitrarily.
    Define $\regg \in W^{1,p}(J;X) \cap L^p(J; D(A))$ and $\reggtau \in S_\tau(D(A))$ by
    \begin{equation}
        \begin{cases}
            \partial_t \regg(t) + A\regg(t) = \regdelta(t) v, & t \in J, \\ 
            \regg(0) = 0,
        \end{cases}
    \end{equation}
    and 
    \begin{equation}
        \begin{dcases}
            B_\tau(\reggtau,\varphi_\tau) = \int_J \dual*{\regdelta v}{\varphi_\tau} dt, & \forall \varphi_\tau \in S_\tau(X'), \\
            (\reggtau)^{0,-} = 0,
        \end{dcases}
    \end{equation}
    respectively.
    Then, there exists $\alpha > 1/p'$ that satisfies
    \begin{equation}
        \Lpnorm{\weight^\alpha A(\regg - \reggtau)}{p}{J}{X} \le C \tau^{\alpha-\frac{1}{p'}} \norm{v}{X}.
    \end{equation}      
\end{proposition}

\subsection{Weighted interpolation error estimate}

Let $\err=\regg-\reggtau$, $\interr=\regg-I_\tau \regg$, and $\res = \err-\interr = I_\tau \regg - \reggtau \in S_\tau(D(A))$,
where $\regg$ and $\reggtau$ are as in \cref{prop:weighted-error-X}.
In this section, we show the following interpolation error estimate.

\begin{lemma}\label{lem:weighted-int-err}
    Let $v \in X$, $p \in (1,\infty)$, and $\alpha < 1 + \frac{1}{p'}$.
    Then, we have
    \begin{equation}
        \Lpnorm{\weight^\alpha A\interr}{p}{J}{X} 
        \le C \tau^{\alpha-\frac{1}{p'}} \norm{v}{X},
        \label{eq:weighted-int-err}
    \end{equation}
    where $C$ is independent of $\tau$, $f$, $\tiln$, $\tilt$, and $v$.
\end{lemma}

We first show the local interpolation error estimate.
Recall that $\regg = \reggtau \equiv 0$ in $J_n$ if $n \le \tiln-1$.

\begin{lemma}\label{lem:int-err-loc}
    Let $v \in X$, $p \in [1,\infty)$, and $n \ge \tiln$. 
    Then, 
    \begin{multline}
        \tau^{-1} \Lpnorm{\interr}{p}{J_n}{X} + \Lpnorm{\partial_t \interr}{p}{J_n}{X} + \Lpnorm{A \interr}{p}{J_n}{X} \\ 
        \le \begin{cases}
            C \tau^{-1+\frac{1}{p}} \norm{v}{X}, & n \le \tiln + 1, \\
            C \tau^{1+\frac{1}{p}} (t_{n-1} - t_\tiln)^{-2} \norm{v}{X}, & n \ge \tiln + 2,
        \end{cases}
    \end{multline}
    where $C$ is independent of $\tau$, $\tiln$, $\tilt$, $n$, and $v$.
\end{lemma}

\begin{proof}
    We may assume $p>1$. The estimates for $p=1$ are obtained by the H\"older inequality.

    (i) Let $\tiln \le n \le \tiln+1$. Then, we have
    \begin{multline}
        \tau^{-1} \Lpnorm{\interr}{p}{J_n}{X} 
        + \Lpnorm{\partial_t \interr}{p}{J_n}{X} 
        + \Lpnorm{A \interr}{p}{J_n}{X} \\ 
        \le C \tau \paren*{ \Lpnorm{\partial_t^2 \regg}{p}{J}{X} 
        + \Lpnorm{\partial_t A \regg}{p}{J}{X}  } 
    \end{multline}
    by \eqref{eq:int-err}. 
    Here, the assumption $r \ge 1$ is used for the estimate of $\interr$.
    Notice that $\partial_t \regg$ satisfies the parabolic equation
    \begin{equation}
        \begin{cases}
            \partial_t (\partial_t \regg) (t) + A (\partial_t \regg) (t) = \partial_t \regdelta(t) v, & t \in J, \\ 
            (\partial_t \regg) (0) = 0,
        \end{cases}
    \end{equation}
    since $\regdelta \equiv 0$ in the neighborhood of $t=0$ and so is $\regg$ by the uniqueness. 
    Therefore, by the maximal regularity for this equation and \eqref{eq:delta-norm}, we have
    \begin{equation}
        \Lpnorm{\partial_t^2 \regg}{p}{J}{X} 
            + \Lpnorm{\partial_t A \regg}{p}{J}{X}
            \le C \Lpnorm{v \partial_t \regdelta}{p}{J}{X}
            \le C \tau^{-1-\frac{1}{p'}} \norm{v}{X},
    \end{equation}
    which implies
    \begin{equation}
        \tau^{-1} \Lpnorm{\interr}{p}{J_n}{X} 
        + \Lpnorm{\partial_t \interr}{p}{J_n}{X} 
        + \Lpnorm{A \interr}{p}{J_n}{X} 
        \le C \tau^{-1+\frac{1}{p}} \norm{v}{X},
    \end{equation}
    and we obtain the desired estimate.

    (ii) Let $n \ge \tiln+2$.
    Again by \eqref{eq:int-err}, we have
    \begin{multline}
        \tau^{-1} \Lpnorm{\interr}{p}{J_n}{X} 
        + \Lpnorm{\partial_t \interr}{p}{J_n}{X} 
        + \Lpnorm{A \interr}{p}{J_n}{X} \\
        \begin{aligned}
            &\le C \tau \paren*{ \Lpnorm{\partial_t^2 \regg}{p}{J_n}{X} 
            + \Lpnorm{\partial_t A \regg}{p}{J_n}{X}  } \\ 
            &\le C \tau^{1+\frac{1}{p}} \paren*{ \Lpnorm{\partial_t^2 \regg}{\infty}{J_n}{X} 
            + \Lpnorm{\partial_t A \regg}{\infty}{J_n}{X}  }.
        \end{aligned}
    \end{multline}
    Now, owing to $t \in J_n$ and $n \ge \tiln + 2$ (i.e., $t \ge t_\tiln$), we have
    \begin{equation}
        \partial_t^2 \regg(t) = -\partial_t A \regg(t) 
        = \int_{t_{\tiln-1}}^{t_\tiln} A^2 e^{-(t-s)A} v \regdelta(s) ds
    \end{equation}
    by the Duhamel formula.
    Therefore, we have 
    \begin{equation}
        \norm{\partial_t^2 \regg(t)}{X} = \norm{\partial_t A \regg(t)}{X} 
        \le C \int_{t_{\tiln-1}}^{t_\tiln} (s-t)^{-2} \norm{v}{X} \regdelta(s) ds
        \le C (t-t_\tiln)^{-2} \norm{v}{X},
    \end{equation}
    which implies
    \begin{equation}
        \Lpnorm{\partial_t^2 \regg}{\infty}{J_n}{X} 
        = \Lpnorm{\partial_t A \regg}{\infty}{J_n}{X} 
        \le C (t_{n-1} - t_\tiln)^{-2} \norm{v}{X}.
    \end{equation}
    Summarizing the above estimates, we have the desired estimate
    \begin{equation}
        \tau^{-1} \Lpnorm{\interr}{p}{J_n}{X} 
        + \Lpnorm{\partial_t \interr}{p}{J_n}{X} 
        + \Lpnorm{A \interr}{p}{J_n}{X} 
        \le C \tau^{1+\frac{1}{p}} (t_{n-1} - t_\tiln)^{-2} \norm{v}{X},
    \end{equation}
    and hence we complete the proof.
\end{proof}

\begin{proof}[Proof of \cref{lem:weighted-int-err}]
    Notice that the weight function $\weight$ satisfies
    \begin{equation}
        \weight|_{J_n} \le 
        \begin{cases}
            C \tau , & \tiln \le n \le \tiln + 1, \\ 
            C (t_{n-1} - t_\tiln), & n \ge \tiln + 2.
        \end{cases}
        \label{eq:weight-Jn}
    \end{equation}
    Indeed, this is clear when $\tiln \le n \le \tiln+1$, and when $n \ge \tiln + 2$,
    we have, for $t \in J_n$,
    \begin{align}
        \weight(t) 
        &\le C \sqrt{(t - t_{n-1})^2 + (t_{n-1} - t_\tiln)^2 + (t_\tiln-\tilt)^2 + \tau^2} \\ 
        &\le C \sqrt{(t_{n-1} - t_\tiln)^2 + 3\tau^2} \\ 
        &\le C (t_{n-1} - t_\tiln)
    \end{align}
    since $\tau \le C \tau_{n-1} \le C(t_{n-1} - t_\tiln)$.
    Therefore, together with \cref{lem:int-err-loc}, we have
    \begin{equation}
        \Lpnorm{\weight^\alpha A \interr}{p}{J_n}{X} \le 
        \begin{cases}
            C \tau^{\alpha-1+\frac{1}{p}} \norm{v}{X}, & \tiln \le n \le \tiln+1, \\
            C \tau^{1+\frac{1}{p}} (t_{n-1} - t_\tiln)^{\alpha-2} \norm{v}{X}, & n \ge \tiln + 2,
        \end{cases}
    \end{equation}
    which implies
    \begin{align}
        \Lpnorm{\weight^\alpha A \interr}{p}{J}{X}^{p}
        &\le C \paren*{ \tau^{\alpha-1+\frac{1}{p}} \norm{v}{X} }^{p}
        + C \paren*{ \tau^{1+\frac{1}{p}} \norm{v}{X} }^{p} \sum_{n \ge \tiln+2} (t_{n-1} - t_\tiln)^{(\alpha-2)p}.
        \label{eq:weighted-int-err-1}
    \end{align}
    Here, in general, for $l-m \ge 2$ and $\beta>1$, we have
    \begin{equation}
        \sum_{n \ge l} (t_{n-1}-t_m)^{-\beta} \tau 
        \le C \sum_{n \ge l} (t_{n-1}-t_m)^{-\beta} \tau_n
        \le C \int_{t_{l-1}}^\infty (t-t_m)^{-\beta} dt
        \le C (t_{l-1} - t_m)^{-\beta+1}
        \label{eq:summation-t_n}
    \end{equation}
    by the quasi-uniformity of the mesh.
    We thus obtain
    \begin{equation}
        \sum_{n \ge \tiln+2} (t_{n-1} - t_\tiln)^{(\alpha-2)p}
        = \tau^{-1} \sum_{n \ge \tiln+2} (t_{n-1} - t_\tiln)^{(\alpha-2)p} \tau
        \le C \tau^{(\alpha-2)p},
    \end{equation}
    since the assumption $\alpha<1+\frac{1}{p'}$ yields $(\alpha-2)p < -1$.
    Hence we obtain \eqref{eq:weighted-int-err} with the aid of \eqref{eq:weighted-int-err-1}.
\end{proof}

\section{Proof of discrete maximal regularity}
\label{sec:proof}

This section is devoted to the proof of \cref{thm:dmr}.

\subsection{Proof of \texorpdfstring{\cref{prop:weighted-error-X}}{Proposition 3.1}}

We first show \eqref{eq:dmr-Au} when $u_0 = 0$ by showing \cref{prop:weighted-error-X}.
We begin with the local estimate for $A\res$. 
This is the most important part for the proof of \cref{prop:weighted-error-X}. 

\begin{lemma}\label{lem:A-res-L^infty}
    Let $n \ge \tiln$ and $v \in X$. Then, we have 
    \begin{equation}
        \label{eq:A-res-L^infty}
        \Lpnorm{A \res}{\infty}{J_n}{X} \le 
        \begin{cases}
            C \tau^{-1} \norm{v}{X}, & \tiln \le n \le \tiln + 3, \\ 
            C \tau (t_{n-2} - t_\tiln)^{-2}  \norm{v}{X}, & n \ge \tiln + 4,
        \end{cases}
    \end{equation}
    where $C$ is independent of $\tau$,  $\tiln$, $\tilt$, $n$, and $v$.  
\end{lemma}

\begin{proof}
    Let $Z^n_i \in X$ be such that $\res|_{J_n} = \sum_{i=0}^r Z^n_i \phi^n_i$, where $\phi^n_i \in \cP^r(J_n)$ is as in \eqref{eq:local-basis}.
    Then, it is clear that 
    \begin{equation}
        \Lpnorm{A \res}{\infty}{J_n}{X}
        \le \sum_{i=0}^r \norm{AZ^n_i}{X}
        \label{eq:A-res-L^infty-1}
    \end{equation}
    holds.

    We first notice that $\res$ satisfies
    \begin{equation}
        B_\tau(\res,\varphi_\tau) = - \int_J \dual{A\interr}{\varphi_\tau} dt,
        \qquad \forall \varphi_\tau \in S_\tau(X')
    \end{equation}
    by the definition of $I_\tau$ and the Galerkin orthogonality $B_\tau(\err,\varphi_\tau) = 0$ (cf.~\cite[p.~208]{MR2249024}).
    Therefore, by \cref{cor:Duhamel-DG} and $(\res)^{0,-}=0$, we have    
    \begin{multline}
        AZ^n_i 
        = - \sum_{m=\tiln}^{n-1} AR_{i,0}(\tau_n A) \paren*{ \prod_{l=m+1}^{n-1} R_{r,0}(\tau_l A) } \sum_{j=0}^r R_{r,j}(\tau_m A) \int_{J_m} A\interr(t) \phi^m_j(t) dt \\ 
        - \sum_{j=0}^r R_{i,j}(\tau_n A) \int_{J_n} A\interr(t) \phi^n_j(t) dt
    \end{multline}
    for $n \ge \tiln$.
    Observe that, for $m \le n-2$ and $\delta \in \paren*{\delta_0,\pi/2}$, 
    \begin{multline}
        A^2 R_{i,0}(\tau_n A) \paren*{ \prod_{l=m+1}^{n-1} R_{r,0}(\tau_l A) } R_{r,j}(\tau_m A) \\ 
        = \frac{1}{2\pi\sqrt{-1}} \int_{\Gamma_\delta} \lambda^2 R_{i,0}(\tau_n \lambda) \paren*{ \prod_{l=m+1}^{n-1} R_{r,0}(\tau_l \lambda) } R_{r,j}(\tau_m \lambda) (\lambda I - A)^{-1} d\lambda
        \label{eq:Dunford-DG}
    \end{multline}
    holds, where $\delta_0$ is as in \eqref{eq:resolvent}.
    Here, $\Gamma_\delta = \partial \Sigma_\delta$ and it is oriented so that $\operatorname{Im} \lambda$ decreases.
    We then decompose $AZ^n_i$ into three parts:
    \begin{align}
        AZ^n_i &= G_1 + G_2 + G_3, \\ 
        G_1 &= -\sum_{m=\tiln}^{n-2} A^2 R_{i,0}(\tau_n A) \paren*{ \prod_{l=m+1}^{n-1} R_{r,0}(\tau_l A) } \sum_{j=0}^r R_{r,j}(\tau_m A) \int_{J_m} \interr(t) \phi^m_j(t) dt, \\
        G_2 &= -A^2 R_{i,0}(\tau_n A) \sum_{j=0}^r R_{r,j}(\tau_{n-1} A) \int_{J_{n-1}} \interr(t) \phi^{n-1}_j(t) dt, \\ 
        G_3 &= -\sum_{j=0}^r A R_{i,j}(\tau_n A) \int_{J_n} A\interr(t) \phi^n_j(t) dt.
    \end{align}
    Here, we set $G_1=0$ if $n \le \tiln+1$.
    We address $G_2$ and $G_3$. 
    By \cref{lem:expression-DG}, the partial fraction decomposition, and the resolvent estimate \eqref{eq:resolvent}, we have
    \begin{equation}
        \norm{R_{i,j}(\tau A)}{\cL(X)} \le C, \qquad \norm{A R_{i,j}(\tau A)}{\cL(X)} \le C \tau^{-1}
        \label{eq:resolvent-Rij}
    \end{equation}
    for $\tau>0$. 
    Hence, by $\norm{\phi^m_j}{L^\infty(J_m)} = 1$ and \cref{lem:int-err-loc}, we have
    \begin{align}
        \norm{G_2}{X} &\le C \tau^{-2} \Lpnorm{\interr}{1}{J_{n-1}}{X}
        \le \begin{cases}
            C \tau^{-1} \norm{v}{X}, & n \le \tiln + 2, \\
            C \tau (t_{n-2} - t_\tiln)^{-2} \norm{v}{X}, & n \ge \tiln + 3,
        \end{cases} \\ 
        \norm{G_3}{X} &\le C \tau^{-1} \Lpnorm{A \interr}{1}{J_n}{X}
        \le \begin{cases}
            C \tau^{-1} \norm{v}{X}, & n \le \tiln + 1, \\
            C \tau (t_{n-1} - t_\tiln)^{-2} \norm{v}{X}, & n \ge \tiln + 2,
        \end{cases} 
    \end{align}
    which imply
    \begin{equation}
        \norm{G_2}{X} + \norm{G_3}{X} \le 
        \begin{cases}
            C \tau^{-1} \norm{v}{X}, & n \le \tiln + 2, \\
            C \tau (t_{n-2} - t_\tiln)^{-2} \norm{v}{X}, & n \ge \tiln + 3.
        \end{cases}
        \label{eq:G2G3}
    \end{equation}

    We address $G_1$. First we assume $n \le \tiln + 3$.
    In this case, by \eqref{eq:resolvent-Rij}, we obtain
    \begin{equation}
        \label{eq:G1-0}
        \norm{G_1}{X} 
        \le C \tau^{-2} \sum_{m=\tiln}^{n-2} \Lpnorm{\interr}{1}{J_m}{X} 
        \le C \tau^{-1} \norm{v}{X}.
    \end{equation}
    We then assume $n \ge \tiln +4$. 
    By \eqref{eq:Dunford-DG} and \ref{eq:rational-DG}, we have
    \begin{align}
        \norm{G_1}{X} 
        &\le C \sum_{m=\tiln}^{n-2} \int_{\Gamma_\delta} |\lambda|^2 \paren*{ \prod_{l=m+1}^n \frac{1}{1 + C \tau_l |\lambda|} } \frac{|d\lambda|}{|\lambda|} \Lpnorm{\interr}{1}{J_m}{X} \\ 
        &\le C \sum_{m=\tiln}^{n-2} \int_0^\infty x \paren*{ \prod_{l=m+1}^n \frac{1}{1 + C \tau_l x} } dx \Lpnorm{\interr}{1}{J_m}{X}.
    \end{align}
    We address the denominator (cf.~\cite[p.~1321]{MR1620144}). 
    Notice that 
    \begin{equation}
        \prod_{l=m+1}^n (1 + C \tau_l x) \ge 1 + C x^3 \sum_{m+1 \le l_1 < l_2 < l_3 \le n} \tau_{l_1} \tau_{l_2} \tau_{l_3}
        \label{eq:prod}
    \end{equation}
    for $x \ge 0$. We then prove
    \begin{equation}
        \sum_{m+1 \le l_1 < l_2 < l_3 \le n} \tau_{l_1} \tau_{l_2} \tau_{l_3} 
        \ge C \paren*{ \sum_{l=m+1}^n \tau_l }^3 = C (t_n-t_m)^3.
        \label{eq:tau-cube}
    \end{equation}

    (a) Assume $\tau \le (t_n-t_m)/4$. Then, noticing that
    \begin{equation}
        \paren*{ \sum_{l=m+1}^n \tau_l }^3
        = \sum_{l=m+1}^n \tau_l^3 + 3 \sum_{l_1 \ne l_2} \tau_{l_1}^2 \tau_{l_2} + 6 \sum_{m+1 \le l_1 < l_2 < l_3 \le n} \tau_{l_1} \tau_{l_2} \tau_{l_3} 
    \end{equation}
    holds and the assumption $\tau_l \le \tau \le (t_n-t_m)/4$ yields
    \begin{align}
        \sum_{l=m+1}^n \tau_l^3 
        &\le \frac{(t_n-t_m)^2}{16} \sum_{l=m+1}^n \tau_l 
        = \frac{(t_n-t_m)^3}{16}, \\
        \sum_{l_1 \ne l_2} \tau_{l_1}^2 \tau_{l_2} 
        &\le \frac{(t_n-t_m)}{4} \sum_{l_1} \tau_{l_1} \sum_{l_2} \tau_{l_2}
        = \frac{(t_n-t_m)^3}{4},
    \end{align}
    we have
    \begin{equation}
        \paren*{ \sum_{l=m+1}^n \tau_l }^3
        \le \paren*{\frac{1}{16} + \frac{3}{4}}(t_n-t_m)^3 + 6 \sum_{m+1 \le l_1 < l_2 < l_3 \le n} \tau_{l_1} \tau_{l_2} \tau_{l_3} .
    \end{equation}
    Therefore, we obtain \eqref{eq:tau-cube}.

    (b) Assume $\tau \ge (t_n-t_m)/4$.
    In this case, by the quasi-uniformity, we easily obtain \eqref{eq:tau-cube} since
    \begin{equation}
        \sum_{m+1 \le l_1 < l_2 < l_3 \le n} \tau_{l_1} \tau_{l_2} \tau_{l_3} 
        \ge \tau_{m+1} \tau_{m+2} \tau_{m+3}
        \ge C \tau^3
        \ge C (t_n - t_m)^3.
    \end{equation}

    From \eqref{eq:prod} and \eqref{eq:tau-cube}, we have
    \begin{equation}
        \prod_{l=m+1}^n (1 + C \tau_l x) \ge 1 + C (t_n - t_m)^3 x^3 ,
    \end{equation}
    which implies
    \begin{equation}
        \int_0^\infty x \paren*{ \prod_{l=m+1}^n \frac{1}{1 + C \tau_l x} } dx 
        \le C \int_0^\infty \frac{x}{1 + C (t_n - t_m)^3 x^3 } dx 
        \le C (t_n - t_m)^{-2} .
    \end{equation}
    Hence, from \cref{lem:int-err-loc} and $\tiln<\tiln+1<n$, we obtain
    \begin{align}
        \label{eq:G1-1}
        \norm{G_1}{X}
        &\le C \sum_{m=\tiln}^{n-2} (t_n-t_m)^{-2} \Lpnorm{\interr}{1}{J_m}{X} \\
        &\le C \tau \norm{v}{X} \sum_{m=\tiln}^{\tiln+1} (t_n-t_m)^{-2} 
        + C \tau^3 \norm{v}{X} \sum_{m=\tiln+2}^{n-2} (t_n-t_m)^{-2}  (t_{m-1} - t_\tiln)^{-2} \notag \\ 
        &\le C \tau (t_n - t_{\tiln+1})^{-2} \norm{v}{X} 
        + C \tau^3 \norm{v}{X} \sum_{m=\tiln+2}^{n-2} (t_n-t_m)^{-2}  (t_{m-1} - t_\tiln)^{-2} . \notag 
    \end{align}
    Notice that the last term is well-defined when $n \ge \tiln + 4$.

    We then address the last summation. Observe that 
    \begin{align}
        \label{eq:G1-2}
        \sum_{m=\tiln+2}^{n-2} (t_n-t_m)^{-2} (t_{m-1}-t_\tiln)^{-2}
        &\le C \tau^{-1} \sum_{m=\tiln+2}^{n-2} (t_n-t_m)^{-2} (t_{m-1}-t_\tiln)^{-2} \tau_m \\ 
        &\le C \tau^{-1} \sum_{m=\tiln+2}^{n-2} \int_{J_m} (t_n-t)^{-2} (t-t_\tiln)^{-2} dt \notag\\ 
        &= C \tau^{-1} \int_{t_{\tiln+1}}^{t_{n-2}} (t_n-t)^{-2} (t-t_\tiln)^{-2} dt \notag
    \end{align}
    since $t_n - t = (t_n - t_m) + (t_m - t) \le C (t_n - t_m)$ and similarly $t-t_\tiln \le C(t_{m-1} - t_\tiln)$ for $t \in J_m$.
    Set $s = (t-t_\tiln)/(t_n - t_\tiln)$ so that $t-t_\tiln = (t_n - t_\tiln)s$ and $t_n-t = (t_n-t_\tiln)(1-s)$.
    Then,
    \begin{equation}
        \int_{t_{\tiln+1}}^{t_{n-2}} (t_n-t)^{-2} (t-t_\tiln)^{-2} dt 
        = (t_n - t_\tiln)^{-3} \int_{s_1}^{s_2} s^{-2} (1-s)^{-2} ds,
        \label{eq:G1-3}
    \end{equation}
    where $s_1 = (t_{\tiln+1}-t_\tiln)/(t_n - t_\tiln)$ and $s_2 = (t_{n-2}-t_\tiln)/(t_n-t_\tiln)$.
    We will show that
    \begin{equation}
        I \coloneqq \int_{s_1}^{s_2} s^{-2} (1-s)^{-2} ds \le C \tau^{-1} (t_n - t_\tiln)
        \label{eq:G1-4}
    \end{equation} 
    holds. If $s_2<1/2$, then we have
    \begin{equation}
        I 
        \le \int_{s_1}^{1/2} s^{-2} (1-s)^{-2} ds 
        \le 4 \int_{s_1}^{1/2} s^{-2} ds
        \le 4s_1^{-1} 
        \le C \tau^{-1} (t_n - t_\tiln)
    \end{equation}
    and obtain \eqref{eq:G1-4}. Similarly, if $s_1 > 1/2$, we have
    \begin{equation}
        I 
        \le \int_{1/2}^{s_2} s^{-2} (1-s)^{-2} ds 
        \le 4 (1-s_2)^{-1}
        \le C \tau^{-1} (t_n - t_\tiln).
    \end{equation}
    If $s_1 \le 1/2 \le s_2$, we have
    \begin{equation}
        I 
        = \paren*{ \int_{s_1}^{1/2} + \int_{1/2}^{s_2} } s^{-2} (1-s)^{-2} ds 
        \le C \tau^{-1} (t_n - t_\tiln)
    \end{equation}
    by the above estimates. In any cases, we obtain  \eqref{eq:G1-4}.

    Finally, by the quasi-uniformity of the mesh and $n \ge \tiln + 4$, we have
    \begin{equation}
        t_{n-2} - t_\tiln \ge t_{\tiln+2} - t_\tiln \ge C\tau,
    \end{equation}
    which yields
    \begin{equation}
        t_n - t_{\tiln+1}
        = t_{n-2} - t_\tiln + (\tau_{n-1}+\tau_n-\tau_{\tiln+1}) 
        \le t_{n-2} - t_\tiln + 2\tau 
        \le C(t_{n-2} - t_\tiln).
    \end{equation}
    Therefore, together with \eqref{eq:G1-1}, \eqref{eq:G1-2}, \eqref{eq:G1-3}, and \eqref{eq:G1-4}, we have
    \begin{equation}
        \norm{G_1}{X}
        \le C \tau (t_{n-2} - t_\tiln)^{-2} \norm{v}{X} 
        \label{eq:G1}
    \end{equation}
    when $n \ge \tiln+4$. 
    Summarizing \eqref{eq:A-res-L^infty-1}, \eqref{eq:G2G3},  \eqref{eq:G1-0}, and \eqref{eq:G1},  we obtain the desired estimate \eqref{eq:A-res-L^infty} and hence complete the proof.
\end{proof}

We are ready to show \cref{prop:weighted-error-X}, which yields the discrete maximal regularity with $u_0=0$ together with \cref{lem:dmr-to-weighted}.

\begin{proof}[Proof of \cref{prop:weighted-error-X}]
    In view of \cref{lem:weighted-int-err}, it suffices to show
    \begin{equation}
        \Lpnorm{\weight^\alpha A\res}{p}{J}{X} 
        \le C \tau^{\alpha-\frac{1}{p'}} \norm{v}{X}.
        \label{eq:weighted-res}
    \end{equation}
    By a similar way to show \eqref{eq:weight-Jn}, we have 
    \begin{equation}
        \weight|_{J_n} \le C (t_{n-2} - t_\tiln)
    \end{equation} 
    for $n \ge \tiln+4$.
    Thus, thanks to \cref{lem:A-res-L^infty}, we have 
    \begin{equation}
        \Lpnorm{\weight^\alpha A\res}{p}{J}{X} 
        \le C \tau^{\alpha - 1 + \frac{1}{p}} \norm{v}{X} 
        + C \tau^{1+\frac{1}{p}} \bracket*{ \sum_{n \ge \tiln+4} (t_{n-2} - t_\tiln)^{(\alpha-2)p} }^{1/p} \norm{v}{X}. 
        \label{eq:weighted-res-1}
    \end{equation}
    Fix $\alpha \in \paren*{\frac{1}{p'}, \, \frac{1}{p'}+1}$ arbitrarily so that $(\alpha-2)p< -1$.
    Then, owing to \eqref{eq:summation-t_n}, we obtain 
    \begin{equation} 
        \bracket*{ \sum_{n \ge \tiln+4} (t_{n-2} - t_\tiln)^{(\alpha-2)p} }^{1/p} 
        = \bracket*{ \tau^{-1} \sum_{n \ge \tiln+3} (t_{n-1} - t_\tiln)^{(\alpha-2)p} \tau }^{1/p}
        \le C \tau^{\alpha-2} .
        \label{eq:T1} 
    \end{equation} 
    Therefore, by \eqref{eq:weighted-res-1} and \eqref{eq:T1}, we obtain \eqref{eq:weighted-res} for $\alpha \in \paren*{\frac{1}{p'}, \, \frac{1}{p'}+1}$.
    Hence we complete the proof of \cref{prop:weighted-error-X}.
\end{proof}

\begin{remark}[On the assumption $r \ge 1$]\label{rem:r}
    We assumed $r \ge 1$ in \cref{thm:dmr} and it plays an essential role in the above proof.
    When $r=0$, the local interpolation error estimate for $\interr$ in \cref{lem:int-err-loc} is replaced by 
    \begin{equation}
        \Lpnorm{\interr}{1}{J_n}{X}
        \le C \tau \Lpnorm{\partial_t \regg}{1}{J_n}{X}
        \le \begin{cases}
            C \tau \norm{v}{X}, & n \le \tiln+1, \\
            C \tau^2 (t_{n-1}-t_\tiln)^{-1} \norm{v}{X}, & n \ge \tiln + 2.
        \end{cases}
    \end{equation}
    This affects the estimate of $G_1$ in the above proof.
    Indeed, the estimate \eqref{eq:G1-1} is replaced by 
    \begin{align}
        \norm{G_1}{X}
        &\le C \tau \norm{v}{X} \sum_{m=\tiln}^{\tiln+1} (t_n-t_m)^{-2} 
        + C \tau^2 \norm{v}{X} \sum_{m=\tiln+2}^{n-2} (t_n-t_m)^{-2}  (t_{m-1} - t_\tiln)^{-1} ,
    \end{align}
    and the last summation will give a logarithmic factor.
    Hence our approach does not cover the case of $r=0$.
\end{remark}

We now establish \eqref{eq:dmr-Au} when $u_0 = 0$ thanks to \cref{lem:dmr-to-weighted} and \cref{prop:weighted-error-X}.

\subsection{Estimates for the initial value problem}

We show the estimate \eqref{eq:dmr-Au} when $u_0 \ne 0$.
It suffices to consider the case $f \equiv 0$.

\begin{proof}[Proof of \eqref{eq:dmr-Au} with $f \equiv 0$]
    Recall the real interpolation space $\initialspace$ is characterized by
    \begin{align}
        \initialspace &= \Set*{ u \in X \given \int_0^\infty \abs*{t^{-1} K(t,u)}^p dt < \infty}, \\
        \norm{u}{\initialspace} &= \paren*{ \int_0^\infty \abs*{t^{-1} K(t,u)}^p dt }^{1/p} 
    \end{align}
    where 
    \begin{equation}
        K(t,x) \coloneqq \inf_{x=a+b,\, a\in X, b \in D(A)} \paren*{ \norm{a}{X} + t \norm{Ab}{X} }.
    \end{equation}
    Now let $u_\tau \in S_\tau(D(A))$ be the solution of \eqref{eq:DG} with $f \equiv 0$.
    We will show the a priori estimate
    \begin{equation}
        \Lpnorm{Au_\tau}{p}{J}{X} \le C \norm{u_0}{\initialspace}.
        \label{eq:dmr-initial}
    \end{equation}
    For $u_0 \in \initialspace$, let $a \in X$ and $b \in D(A)$ be such that $u_0 = a + b$. 
    Let moreover $u_\tau^x \in S(D(A))$ ($x=a,b$) be the solution of \eqref{eq:DG} with $f \equiv 0$ and $(u_\tau^x)^{0,-} = x$, and finally we set
    \begin{equation}
        u_\tau^x|_{J_n} = \sum_{i=0}^r U^{x,n}_i \phi^n_i, \qquad U^{x,n}_i \in X
    \end{equation}
    for each $n$.
    Then, by \cref{cor:Duhamel-DG}, we have
    \begin{equation}
        U^{x,n}_i 
        = R_{i,0}(\tau_n A) \paren*{ \prod_{l=1}^{n-1} R_{r,0}(\tau_l A) } x.
    \end{equation}
    Therefore, by the same way as in the proof of \cref{lem:A-res-L^infty} (or \cite[Theorem~5.1]{MR1620144}), we obtain
    \begin{equation}
        \Lpnorm{Au_\tau^a}{\infty}{J_n}{X} \le C t_n^{-1} \norm{a}{X}, \qquad 
        \Lpnorm{Au_\tau^b}{\infty}{J_n}{X} \le C \norm{A b}{X},
    \end{equation}
    which imply 
    \begin{equation}
        \norm{Au_\tau^a(t)}{X} \le C t^{-1} \norm{a}{X}, \qquad 
        \norm{Au_\tau^b(t)}{X} \le C \norm{A b}{X}
    \end{equation}
    for each $t \in J_n$.
    Since $a$ and $b$ are arbitrary, we have
    \begin{equation}
        \norm{Au_\tau(t)}{X} \le C t^{-1} K(t,u_0), \qquad \forall t \in J_n,
    \end{equation}
    which yields
    \begin{equation}
        \Lpnorm{Au_\tau}{p}{J_n}{X} \le C \paren*{ \int_{J_n} \abs*{t^{-1} K(t,u_0)}^p dt }^{1/p}.
    \end{equation}
    Hence, summing this up with respect to $n$, we establish \eqref{eq:dmr-initial}.
\end{proof}

Summarizing the above arguments, we eventually establish \eqref{eq:dmr-Au}, and hence we complete the proof of \cref{thm:dmr} owing to \cref{lem:dudt,lem:jump}.

\begin{remark}
    The constant $C$ in the main result \eqref{eq:dmr} (here we set $C_\text{DMR}$) depends on the exponent $p$, the degree of polynomials $r$, the constant of the quasi-uniformity of the mesh and that of the norm estimate of $\regdelta$ (i.e., \eqref{eq:delta-norm}).
    The dependency on $r$ stems from the constants appearing in \eqref{eq:rational-DG} and \eqref{eq:int-err},
    and that on $p$ stems from the constant of the maximal regularity \eqref{eq:mr} (we set $C_\text{MR}$) .
    In many cases, $C_\text{MR} \nearrow \infty$ when $p \to 1,\infty$ and thus $C_\text{DMR}$ would also goes to infinity when $p \to 1,\infty$.
\end{remark}

\section{\texorpdfstring{Corollaries}{Corollaries}}
\label{sec:application}

\subsection{One-step version of discrete maximal regularity}

Our estimate \eqref{eq:dmr} implies the existing estimate of discrete maximal regularity as a one-step method.
\begin{corollary}
    Under the same assumptions as in \cref{thm:dmr}, we have
    \begin{multline}
        \paren*{\sum_{n=1}^N \norm*{\frac{u_\tau^{n,-}-u_\tau^{n-1,-}}{\tau_n}}{X}^p \tau_n }^{1/p} 
        + \paren*{\sum_{n=1}^N \norm*{Au_\tau^{n,-}}{X}^p \tau_n }^{1/p} \\
        \le C \paren*{ \Lpnorm{f}{p}{J}{X} + \norm{u_0}{\initialspace} }.
        \label{eq:dmr-classical}
    \end{multline}
\end{corollary}

\begin{proof}
    An elementary relation
    \begin{equation}
        u_\tau^{n,-}-u_\tau^{n-1,-}
        = (u_\tau^{n,-} - u_\tau^{n-1,+}) + \jump{u_\tau}^{n-1}
        = \int_{J_n} \partial_t u_\tau dt + \jump{u_\tau}^{n-1}
    \end{equation}
    implies
    \begin{equation}
        \norm*{\frac{u_\tau^{n,-}-u_\tau^{n-1,-}}{\tau_n}}{X}
        \le \tau_n^{-\frac{1}{p}} \Lpnorm{\partial_t u}{p}{J_n}{X}
        + \norm*{\frac{\jump{u_\tau}^{n-1}}{\tau_n}}{X}.
    \end{equation}
    Moreover, by the inverse inequality, we easily obtain
    \begin{equation}
        \norm*{Au_\tau^{n,-}}{X}
        \le \Lpnorm{Au_\tau}{\infty}{J_n}{X}
        \le C \tau_n^{-\frac{1}{p}} \Lpnorm{Au_\tau}{p}{J_n}{X}.
    \end{equation}
    These estimates imply \eqref{eq:dmr-classical} with the aid of \eqref{eq:dmr}.
\end{proof}

\subsection{Optimal order error estimate}

An optimal order error estimate for \eqref{eq:DG} can be derived by using the discrete maximal regularity.
The proof is almost the same as in \cite[Theorem~9]{MR3606467} and thus we omit the proof.

\begin{corollary}
    Under the same assumptions as in \cref{thm:dmr}, we have
    \begin{equation}
        \norm{u-u_\tau}{L^p(J;X)} \le C \norm{u-I_\tau u}{L^p(J;X)},
    \end{equation}
    where $u$ is the solution of the Cauchy problem \eqref{eq:cauchy}.
    Moreover, if $u \in W^{\tilde{r}+1,p}(J;X)$ with $0 \le \tilde{r} \le r$, then we have 
    \begin{equation}
        \norm{u-u_\tau}{L^p(J;X)} \le C \tau^{\tilde{r}+1} \norm{\partial_t^{\tilde{r}+1} u}{L^p(J;X)}.
    \end{equation}
\end{corollary}

Finally, we consider the fully discrete case. 
Let us consider the parabolic equation
\begin{equation}
    \begin{cases}
        \partial_t u(x,t) - \Delta u(x,t) = f(x,t), & \text{in } \Omega \times J, \\ 
        u(x,t) = 0, & x \in \partial\Omega, \, t \in J, \\
        u(x,0) = u_0(x), & x \in \Omega,
    \end{cases}
    \label{eq:parabolic}
\end{equation}
where $\Omega \subset \bR^d$ is a bounded domain, $J = (0,T)$ is a temporal interval, $f \colon \Omega \times J \to \bR$ is a given function, and $u_0\colon \Omega \to \bR$ is a given initial data.
Let $A = -\Delta$ be the Laplace operator with Dirichlet boundary condition on $L^q(\Omega)$ with $q \in (1,\infty)$ and $D(A) \coloneqq W^{2,q}(\Omega) \cap W^{1,q}_0(\Omega)$ be the domain  of $A$.
Assume the boundary of $\Omega$ is sufficiently smooth. 
Then, the solution of \eqref{eq:parabolic} has maximal regularity estimate
\begin{equation}
    \Lpnorm{\partial_t u}{p}{J}{L^q(\Omega)} + \Lpnorm{A u}{p}{J}{L^q(\Omega)}
    \le C \paren*{ \Lpnorm{f}{p}{J}{L^q(\Omega)} + \norm{u_0}{(L^q(\Omega),D(A))_{1-\frac{1}{p},p}} }
\end{equation}
if $f \in L^p(J,L^q(\Omega))$ and $u_0 \in (L^q(\Omega),D(A))_{1-\frac{1}{p},p}$ with $p,q \in (1,\infty)$.

Assume moreover $\Omega$ is convex for simplicity and consider the finite element approximation of \eqref{eq:parabolic}.
Let $\Set{ \cT_h }_h$ be a family of triangulations and assume it is shape-regular and quasi-uniform, $\Omega_h \subset \Omega$ be the polygonal approximation of $\Omega$ consisting of triangles in $\cT_h$, $X_h \subset H^1_0(\Omega_h)$ be the conforming $P^s$-finite element space associated to $\cT_h$ with $s \ge 1$, and finally $A_h \colon X_h \to X_h$ be the discrete Laplace operator defined by 
\begin{equation}
    \paren{A_h u_h,v_h}_{\Omega_h} = \paren{\nabla u_h, \nabla v_h}_{\Omega_h}, \qquad \forall u_h,v_h \in X_h,
\end{equation} 
where $(\cdot,\cdot)_{\Omega_h}$ is the $L^2$-inner product over $\Omega_h$. 
The operator $A_h$ is invertible and satisfies the estimate
\begin{equation}
    \norm{A_h^{-1} f_h}{L^q(\Omega_h)} \le C \norm{f_h}{L^q(\Omega_h)}
\end{equation}
uniformly in $h$. Indeed, letting $u_h = A_h^{-1}f_h$ and $u = A^{-1} \bar{f_h} \in D(A)$, where $\bar{f_h} \in L^q(\Omega)$ is the extension of $f_h$ by zero outside $\Omega_h$, we obtain 
\begin{equation}
    \norm{u_h}{L^q(\Omega_h)}
    \le \norm{u_h - u}{L^q(\Omega_h)} + \norm{u}{L^q(\Omega_h)}
    \le C (h^2+1) \norm{u}{W^{2,q}(\Omega)}
    \le C \norm{f_h}{L^q(\Omega_h)}
\end{equation}
by the $L^q$-error estimate for the finite element method (see e.g.~\cite{MR2373954}) and the elliptic regularity.
This means $0 \in \rho(A_h)$ uniformly in $h$.

Now, the finite element counterpart of \eqref{eq:parabolic} is the following Cauchy problem to find $u_h \in C^0(\bar{J}; X_h) \cap C^1(J;X_h)$ satisfying
\begin{equation}
    \begin{cases}
        \partial_t u_h(t) + A_h u_h(t) = f_h(t), & t \in J, \\ 
        u_h(0) = u_{0,h},
    \end{cases}
    \label{eq:semi-para}
\end{equation}
where $f_h \in L^p(J;X_h)$ and $u_{0,h} \in X_h$ are given data.
Under the above assumptions, $A_h$ has maximal regularity uniformly in $h$, and thus the solution \eqref{eq:semi-para} satisfies
\begin{multline}
    \Lpnorm{\partial_t u_h}{p}{J}{L^q(\Omega)} + \Lpnorm{A_h u_h}{p}{J}{L^q(\Omega)} \\
    \le C \paren*{ \Lpnorm{f_h}{p}{J}{L^q(\Omega)} + \norm{u_{0,h}}{(X_{h,q}, D(A_{h,q}))_{1-\frac{1}{p},p}} },
\end{multline}
where $X_{h,q}$ and $D(A_{h,q})$ are the finite element space $X_h$ equipped with the norm $\norm{\cdot}{L^q(\Omega_h)}$ and $\norm{A_h \cdot}{L^q(\Omega_h)}$, respectively.
The proof of this estimate is given in \cite{MR2218965} for $u_{0,h} = 0$, and the estimate for general $u_{0,h}$ and $f_h=0$ is obtained by the characterization of the real interpolation space by analytic semigroups (see e.g.~\cite[Proposition~6.2]{MR3753604}).
Hence the operator $A_h$ satisfies the assumptions of \cref{thm:dmr} on the space $X_{h,q}$ for $q \in (1,\infty)$.

Let moreover $X_{\tau,h}^{r,s} \coloneqq S_\tau(X_h)$ be the space of space-time piecewise polynomials associated to the temporal mesh as above.
Then, fully discrete scheme for \eqref{eq:parabolic} is formulated as follows: find $u_{\tau,h} \in X_{\tau,h}^{r,s}$ such that 
\begin{equation}
    B_{\tau,h}(u_{\tau,h},\varphi_{\tau,h}) = \int_J (f_h,\varphi_{\tau,h}) dt + (u_{0,h}, \varphi_{\tau,h}^{0,+}),
    \qquad \forall \varphi_{\tau,h} \in X_{\tau,h}^{r,s},
    \label{eq:DG-para}
\end{equation}
where $f_h \in L^p(J;X_h)$ and $u_{0,h} \in X_h$ are as above and
\begin{multline}
    B_{\tau,h}(v_{\tau,h},\varphi_{\tau,h})
    \coloneqq \sum_{n=1}^N \int_{J_n} \bracket*{ 
        \paren*{\partial_t v_{\tau,h}, \varphi_{\tau,h}}_{\Omega_h} 
        + \paren*{\nabla v_{\tau,h}, \nabla \varphi_{\tau,h}}_{\Omega_h} 
    } dt \\
    + \sum_{n=2}^N \paren*{\jump{v_{\tau,h}}^{n-1}, \varphi_{\tau,h}^{n-1,+}}_{\Omega_h}
    + \paren*{v_{\tau,h}^{0,+}, \varphi_{\tau,h}^{0,+}}_{\Omega_h}
\end{multline}
for $v_{\tau,h}, \varphi_{\tau,h} \in X_{\tau,h}^{r,s}$.
The problem \eqref{eq:DG-para} is nothing but the DG scheme \eqref{eq:DG} with $X=X_{h,q}$ and $A=A_h$.
Hence we have the following estimates.

\begin{corollary}
    Let $u_{\tau,h} \in X_{\tau,h}^{r,s}$ be the solution of \eqref{eq:DG-para}. 
    Then, under the above assumptions, we have the fully discrete maximal regularity
    \begin{multline}
        \label{eq:dmr-full}
        \paren*{ \sum_{n=1}^N \Lpnorm{\partial_t u_{\tau,h}}{p}{J_n}{L^q(\Omega_h)}^p }^{1/p} 
        + \Lpnorm{A_h u_{\tau,h}}{p}{J}{L^q(\Omega_h)} 
        + \paren*{\sum_{n=1}^N \norm*{\frac{\jump{u_{\tau,h}}^{n-1}}{\tau_n}}{X}^p \tau_n }^{1/p}  \\ 
        \le C \paren*{ \Lpnorm{f_h}{p}{J}{L^q(\Omega_h)} + \norm{u_{0,h}}{(X_{h,q}, D(A_{h,q}))_{1-\frac{1}{p},p}} },
    \end{multline}
    for $p,q \in (1,\infty)$, where $C$ is independent of $h$, $\tau$, $f_h$, and $u_{0,h}$.

    Moreover, let $f \in L^p(J;L^q(\Omega))$, $u_0 \in \paren{L^q(\Omega), D(A)}_{1-\frac{1}{p},p}$, and $u$ be the solution of \eqref{eq:parabolic}.
    If $f_h(t) = P_h f(t)$ and $u_{0,h} = P_h u_0$ then we have the error estimate
    \begin{multline}
        \label{eq:error-full}
        \Lpnorm{u-u_{\tau,h}}{p}{J}{L^q(\Omega_h)} \\
        \le C \paren*{
            \Lpnorm{u-I_\tau u}{p}{J}{L^q(\Omega_h)}
            + \Lpnorm{P_h u - u}{p}{J}{L^q(\Omega_h)}
            + \Lpnorm{R_h u - u}{p}{J}{L^q(\Omega_h)}
        },
    \end{multline}
    where $P_h$ and $R_h$ are the $L^2$-orthogonal and the Ritz projection onto $X_h$, respectively.
    In particular, if the solution $u$ has the regularity $u \in W^{\tilde{r}+1,p}(J,L^q(\Omega)) \cap L^p(J; W^{\tilde{s}+1,q}(\Omega))$ for $0 \le \tilde{r} \le r$ and $0 \le \tilde{s} \le s$, then we have
    \begin{equation}
        \label{eq:error-full-ord}
        \Lpnorm{u-u_{\tau,h}}{p}{J}{L^q(\Omega_h)}
        \le C \paren*{
            \tau^{\tilde{r}+1} \Lpnorm{\partial_t^{\tilde{r}+1} u}{p}{J}{L^q(\Omega)}
            + h^{\tilde{s}+1} \Lpnorm{\nabla^{\tilde{s}+1} u}{p}{J}{L^q(\Omega)}
        }.
    \end{equation}
\end{corollary}

\begin{proof}
    The first assertion \eqref{eq:dmr-full} is a consequence of \cref{thm:dmr}. 
    The estimates \eqref{eq:error-full} and \eqref{eq:error-full-ord} can be obtained by the same argument as in \cite[\S 5.3]{MR3606467}.
\end{proof}

\bibliographystyle{plain}
\bibliography{reference}
\end{document}